\documentclass[a4paper,reqno,11pt,oneside]{amsart}
\usepackage[left=2.7cm,right=2.7cm,top=3.5cm,
bottom=3.5cm]{geometry}
\usepackage[colorlinks=true,urlcolor=blue,
citecolor=red,linkcolor=blue,linktocpage,pdfpagelabels,
bookmarksnumbered,bookmarksopen]{hyperref}
\usepackage[english]{babel}
\usepackage{graphicx}
\usepackage{mathrsfs}
\usepackage{amssymb, amsmath, amsfonts, amsthm, mathtools}
\usepackage{xcolor}
\usepackage[shortlabels]{enumitem}
\usepackage{latexsym}

\makeatletter
\newtheorem*{rep@theorem}{\rep@title}
\newcommand{\newreptheorem}[2]{%
\newenvironment{rep#1}[1]{%
 \def\rep@title{#2 \ref{##1}}%
 \begin{rep@theorem}}%
 {\end{rep@theorem}}}
\makeatother

\makeatletter
\newcommand{\proofstep}[2]{%
  \par
  \addvspace{\medskipamount}%
  \noindent\emph{Step #1 #2}\par\nobreak
  \addvspace{\smallskipamount}%
  \@afterheading
}
\makeatother

\allowdisplaybreaks

\makeatletter
\@namedef{subjclassname@2020}{%
  \textup{2020} Mathematics Subject Classification}
\makeatother

\def\Amu{A_\mu}
\def\intzo{\int_{[0, 1]}}
\def\intrn{\int_{\mathbb R^N}}
\def\Deltas{(- \Delta)^s}
\def\X{X_0(\Omega)}
\def\gag{[u]_s}
\def\cns{c_{N, s}}
\def\P{\mathbb P}
\def\H{\mathbb H}

\def\critexp{2_{s_\sharp}^*}
\def\kre{k_{r_\varepsilon}}

\DeclareMathOperator{\spanned}{span}

\numberwithin{equation}{section}

\theoremstyle{plain}
\newtheorem{theorem}{Theorem}[section]
\newreptheorem{theorem}{Theorem}
\theoremstyle{plain}
\newtheorem{prop}[theorem]{Proposition}
\newreptheorem{prop}{Proposition}
\theoremstyle{plain}
\newtheorem{lemma}[theorem]{Lemma}
\newreptheorem{lemma}{Lemma}
\theoremstyle{plain}

\theoremstyle{definition}

\newtheorem{examples}[theorem]{Examples}
\newtheorem{remarks}[theorem]{Remarks}
\newtheorem{corollary}[theorem]{Corollary}
\theoremstyle{definition}
\newtheorem{definition}[theorem]{Definition}
\theoremstyle{definition}
\newtheorem{remark}[theorem]{Remark}
\theoremstyle{definition}

\theoremstyle{plain}

\theoremstyle{definition}

\begin{document}

\renewcommand{\labelenumi}{(\roman{enumi})}

\title[Asymptotically linear problems driven by superposition operators]{Multiple solutions to asymptotically linear problems driven by superposition operators}

\author[Danilo Gregorin Afonso]{Danilo Gregorin Afonso}
\address[Danilo Gregorin Afonso]{Dipartimento di Scienze Pure e Applicate \\
 Università degli Studi di Urbino Carlo Bo \\ Piazza della Repubblica 13, 61029 Urbino, Italy}
\email{danilo.gregorinafonso@uniurb.it}

\author[Rossella Bartolo]{Rossella Bartolo}
\address[Rossella Bartolo]{Dipartimento di Meccanica, Matematica e Management \\ 
Politecnico di Bari \\
Via E. Orabona 4, 70125 Bari, Italy}
\email{rossella.bartolo@poliba.it}

\author[Giovanni Molica Bisci]{Giovanni Molica Bisci}
\address[Giovanni Molica Bisci]{Department of Human Sciences and Quality of Life Promotion \\
Università di Roma San Raffaelle \\ 
Via di Val Cannuta, 00166 Roma, Italy}
\email{giovanni.molicabisci@uniroma5.it}



\begin{abstract}
In this paper, we investigate the existence and multiplicity of weak solutions to problems involving a superposition operator of the type
$$\intzo \Deltas u \ d \mu(s),$$
for a signed measure $\mu$ on the interval of fractional exponents $[0,1]$, when the nonlinearity is subcritical and asymptotically linear at infinity; thus, we deal with a perturbation of the eigenvalue problem for the superposition operator. We use variational tools, extending to this setting well-known results for the classical and the fractional Laplace operators.

\end{abstract}

\maketitle

\section{Introduction}
\label{sec:intro}
In recent years, much attention has been drawn to elliptic problems driven by mixed-order operators, generally consisting of combinations of the Laplacian with its fractional (integral) powers. See, e.g.,  \cite{BiagiDipierroValdinociVecchi2023, BiagiVecchiDipierroValdinoci2020, DipierroPereraSportelliValdinoci2024, MalanchiniMolicaBisciSecchi2025} and the references therein.

Our focus in this work is on asymptotically linear problems driven by the superposition of fractional Laplacians. More precisely, we deal with the operator
\begin{equation*}
    \Amu u \coloneqq \intzo \Deltas u \ d \mu(s),
\end{equation*}
where $\mu = \mu^+ - \mu^-$, with $\mu^+$ and $\mu^-$ nonnegative finite Borel measures in $[0, 1]$, and $\Deltas$ denotes the integral fractional Laplacian, which can be defined, for all $s \in (0, 1)$, by
\begin{equation*}
    \Deltas u(x) = c_{N, s} \intrn \frac{2u(x) - u(x + y) - u(x - y)}{|y|^{N + 2s}} dy, \quad u \in \mathcal S,
\end{equation*}
where $\mathcal S$ is the Schwartz class of smooth, rapidly decaying functions, see \cite{MolicaBisciRadulescuServadei2016}.

Superpositions of fractional Laplacians of different orders were first considered in \cite{CabreSerra2016} and appear in many applications, such as biological models where the individuals of a population are subject to Gaussian and L\'evy flights (see \cite{DipierroValdinoci2021}, \cite{DipierroLippiValdinoci2022}). Since then, operators of the kind $\Amu$, where some terms of smaller order are allowed to have a negative contribution, have drawn much interest, see \cite{DipierroPereraSportelliValdinoci2024existence, DipierroPereraSportelliValdinoci2024}, also under Neumann boundary conditions \cite{DipierroLippiSportelliValdinoci2025Neumann}. 

Let us comment on some particular cases of interest. In what follows, $\delta_s$ denotes the Dirac measure with center at $s \in [0, 1]$.

\begin{examples} {\label{ex:1}}\hfill
\begin{enumerate}
    \item If $\mu = \delta_1$, then $A_\mu=-\Delta$ and we recover well-known results for the Laplacian (see \cite{BartoloBenciFortunato1983});

    \item if $\mu = \delta_s$, with $s \in (0, 1)$, then $A_\mu=(-\Delta)^s$ and we recover well-known results for the integral fractional Laplacian (see \cite{BartoloMolicaBisci2015});

    \item if $\mu = \delta_{s_1} + \delta_{s_2}$, then $\Amu$ is a sum of fractional Laplacians of different orders. Similarly, we could consider $\mu = \delta_1 + \delta_s$, so $A_\mu= -\Delta + (-\Delta)^s$;

    \item if $\mu = \delta_1 - \alpha \delta_s$, with $\alpha > 0$ small, then we have a combination of operators with the ``wrong sign", which is usually an obstacle when dealing with variational problems with operators of mixed orders.
\end{enumerate}
\end{examples}

In order to recover classical variational methods, we make the following structural assumption regarding the signed measure $\mu$: there exist $\bar s \in (0, 1]$ and $\gamma \geq 0$ such that
\begin{align}
    & \mu^+([\bar s, 1]) > 0, \label{eq:hyp_mu_a} \\
    & \mu^-|_{[\bar s, 1]} = 0, \label{eq:hyp_mu_b} \\
    & \mu^-([0, \bar s]) \leq \gamma \mu^+([\bar s, 1]). \label{eq:hyp_mu_c}    
\end{align}

Some further examples are as follows.

\begin{examples}{\label{ex:2}}\hfill
    \begin{enumerate}
        \item Let us consider a convergent series $\sum_{k = 1}^\infty c_k < + \infty$, where either $c_k \geq 0$ for all $k \in \mathbb N$ or, more generally, the condition \eqref{eq:hyp_mu_c} holds, and set, for $1\geq s_0>s_1>s_2>\ldots\geq 0$,    \begin{equation*}
        \Amu u= \sum_{k = 1}^\infty c_k (- \Delta)^{s_k}u;
    \end{equation*}

    \item if $\mu(s) = f(s) \ ds$, for some measurable function $f$, then $\Amu$ is a continuous superposition of fractional operators, as has been considered, e.g, in \cite{CabreSerra2016}:
    \begin{equation*}
        \Amu u = \int_0^1 \Deltas u f(s) \ ds.
    \end{equation*}
    \end{enumerate}
\end{examples}

We are interested in investigating the existence of solutions to asymptotically linear problems driven by the operator $\Amu$. More precisely, we consider the following Dirichlet problem:

\begin{equation}
    \label{eq:main_problem}
    \left\{
    \begin{array}{rcll}
        \Amu u - \bar \lambda u & = & f(x, u) & \quad \text{ in } \Omega \\
        u & = & 0 & \quad \text{ in } \mathbb R^N \setminus \Omega
    \end{array}
    \right.
    ,
\end{equation}
where $\Omega \subset \mathbb R^N$, $N \geq 2$, is an open bounded set with Lipschitz boundary, and $f: \Omega \times \mathbb R \to \mathbb R$ is a Carathéodory function satisfying
\begin{align}
    & \left(x \mapsto \sup_{|t| \leq \tau} |f(x, t)|\right) \in L^\infty(\Omega) \quad \forall \tau > 0, \tag{$f_1$} \label{eq:hyp_f_1} \\
    & \lim_{|t| \to + \infty} \frac{f(x, t)}{t} = 0 \quad \text{ uniformly almost everywhere in } \Omega, \tag{$f_2$} \label{eq:hyp_f_2} \\
    & \lim_{t \to 0} \frac{f(x, t)}{t} = \lambda_0 \in \mathbb R\setminus\{0\} \quad \text{ uniformly almost everywhere in } \Omega. \tag{$f_3$} \label{eq:hyp_f_3}
\end{align}

Problem \eqref{eq:main_problem} is a perturbation of the eigenvalue problem
\begin{equation}
    \label{eq:eigenvalue_problem}
    \left\{
    \begin{array}{rcll}
        \Amu u & = & \lambda u & \quad \text{ in } \Omega \\
        u & = & 0 & \quad \text{ in } \mathbb R^N \setminus \Omega
    \end{array}
    \right.
    ,
\end{equation}
which has been recently been analyzed in detail in \cite{DipierroLippiSportelliValdinoci2025spectral}. It is known that there exists a sequence of eigenvalues (counted with multiplicity)
\begin{equation*}
    0 < \lambda_1 < \lambda_2 \leq \lambda_3 \leq \ldots \leq \lambda_n \leq \ldots \to + \infty;
\end{equation*}
we refer to Section \ref{sec:prelim} for more details on the spectral theory for \eqref{eq:eigenvalue_problem}. Hereafter, we denote by $\sigma(A_\mu)$ the spectrum of $A_\mu$.

Next, we state our main results about the existence and multiplicity of solutions to \eqref{eq:main_problem}. Our approach is variational, and weak solutions are found as critical points of a functional in an appropriate Banach space denoted by $X_0(\Omega)$ (see Section \ref{sec:prelim}).
 
 \begin{prop}{\label{prop:existence_via_saddle_point}}  
Suppose that $\mu = \mu^+ - \mu^-$ satisfies \eqref{eq:hyp_mu_a}-\eqref{eq:hyp_mu_c} and $f$ satisfies \eqref{eq:hyp_f_1}-\eqref{eq:hyp_f_2}. Then, there exists $\gamma_0 > 0$ such that if $\gamma \in [0, \gamma_0]$ and if $\bar \lambda \notin \sigma(\Amu)$, there exists a solution $u \in \X$ to \eqref{eq:main_problem}.
\end{prop}

\begin{theorem}{\label{thm:main}}
Suppose that $\mu = \mu^+ - \mu^-$ satisfies \eqref{eq:hyp_mu_a}-\eqref{eq:hyp_mu_c}, $f$ satisfies \eqref{eq:hyp_f_1}-\eqref{eq:hyp_f_3}, and that $f(x, \cdot)$ is odd almost everywhere in $\Omega$. Furthermore, assume that there exist $h, k \in \mathbb N$, with $h \leq k$, such that
    \begin{equation}
        \label{eq:hyp_Lambda}
        \tag{$\Lambda$}
        \lambda_0 + \bar \lambda < \lambda_h \leq \lambda_k < \bar \lambda.
    \end{equation}
    Then, there exists $\gamma_0 > 0$, depending on $N, \Omega$, $s_\sharp$, and $\bar \lambda$, such that, if $\gamma \in [0, \gamma_0]$ and $\bar \lambda \notin \sigma(\Amu)$, \eqref{eq:main_problem} admits at least $k - h + 1$ pairs of nontrivial weak solutions in $\X$.
\end{theorem}

We will show that Proposition \ref{prop:existence_via_saddle_point} is a direct consequence of the Saddle Point Theorem (see \cite[Theorem 4.6]{Rabinowitz1986}), 
    while the proof of Theorem \ref{thm:main} is based on the application of an abstract critical point theorem in \cite[Theorem 2.9]{BartoloBenciFortunato1983} that we recall in Section \ref{sec:prelim} for the reader's convenience. A recent application of this abstract result in the nonlocal magnetic setting has been given in \cite{BartolodAveniaMolicaBisci2024}.
    
	\begin{remarks}
		\rm{We point out that
			\begin{enumerate}
				\item the  statement in  Theorem \ref{thm:main} holds with slight changes in the proof also if, instead of condition \eqref{eq:hyp_Lambda}, we require $\bar\lambda< \lambda_h \leq \lambda_k < \lambda_0 + \bar\lambda$ (see \cite[Theorem 3.1]{BartoloCandelaSalvatore2012});
				\item if  $\lambda_0$ in \eqref{eq:hyp_f_3} belongs to $\{\pm\infty\}$, then we can reason as in \cite[Remark 3.3]{BartoloCandelaSalvatore2012};
				\item we refer to \cite[Remarks 1.5 and 3.2]{BartoloCandelaSalvatore2012} for some comments about the case $\lambda_0=0$;
				\item in case of resonance, i.e., if $\bar\lambda\in\sigma(A_\mu)$, we can proceed as in \cite[Theorem 1.2]{BartoloCandelaSalvatore2012}, up to adding further assumptions.
		\end{enumerate}}	
	\end{remarks}

Examples \ref{ex:1} and \ref{ex:2} make it clear that our results give rise to several corollaries, many of which are new in more than one case (see, e.g., Corollary \ref{cor:serie}).

In a forthcoming paper, we will consider the case of nonlinear fractional operators of mixed order of $p$-Laplacian type, in the spirit of \cite{BartoloMolicaBisci2017}.

This paper is organized as follows: in Section \ref{sec:prelim} we depict the main aspects of our non-local setting, recall some properties about the spectrum of the superposition operator $A_\mu$  and present some abstract tools; then,  in Section \ref{sec:existence} we prove Proposition \ref{prop:existence_via_saddle_point} and Theorem \ref{thm:main}.	

\section{Preliminaries}
\label{sec:prelim}

\subsection{Definitions and notations}

Let $s \in (0, 1)$. The $s$-Gagliardo seminorm of a measurable function $u: \mathbb R^N \to \mathbb R$ is given by
\begin{equation*}
    [u]_s \coloneqq \left(c_{N, s} \intrn \intrn \frac{|u(x) - u(y)|^2}{|x - y|^{N + 2s}} \ dx \ dy \right)^{\frac{1}{2}},
\end{equation*}
where the normalization constant $c_{N, s}$ is chosen so that 
\begin{equation*}
    \lim_{s \to 1^-} [u]_s = [u]_1 \coloneqq \|\nabla u\|_{L^2(\mathbb R^N)} \quad \text{ and } \quad \lim_{s \to 0^+}  [u]_s=[u]_0 \coloneqq \|u\|_{L^2(\mathbb R^N)}.
\end{equation*}
For more on the integral fractional Laplacian and fractional Sobolev-type spaces, see \cite{DiNezzaPalatucciValdinoci2012, MolicaBisciRadulescuServadei2016}.

We denote by $C_j$, with $j \in \mathbb N$, some positive constants that appear in the calculations, and whose specific values are not of interest.

Finally, the $L^p(\Omega)$ norm is denoted by $\|\cdot\|_p$.
\subsection{Functional setting}
We deal with the operator $A_\mu$ under assumptions \eqref{eq:hyp_mu_a}-\eqref{eq:hyp_mu_c}, so basically the component of the signed measure $\mu$ supported on higher fractional exponents is positive and, if there are negative components of $\mu$, they must be controlled (see \eqref{eq:hyp_mu_c}) by the positive ones. Let us also point out that $\gamma$ in \eqref{eq:hyp_mu_c} is taken sufficiently small and that by \eqref{eq:hyp_mu_a} there exists $s_\sharp\in [\bar s,1]$ such that $\mu^+([s_\sharp,1])>0$.

The next lemma shows that higher exponents in fractional norms control lower exponents by uniform constants.
\begin{lemma}[{\cite[Lemma 2.1]{DipierroPereraSportelliValdinoci2024}}]
    \label{lem:DpPSV_lemma_2.1}
    Let $0 \leq s_1 \leq s_2 \leq 1$. Then, there exists a constant $c = c(N, \Omega)>0$ such that, for any measurable function $u: \mathbb R^N \to \mathbb R$ such that $u = 0$ almost everywhere in $\mathbb R^N \setminus \Omega$, it holds 
    \begin{equation*}
        [u]_{s_1} \leq c [u]_{s_2}.   
    \end{equation*}
\end{lemma}

An appropriate Sobolev-type space in which to consider the weak formulation of problems driven by the operator $\Amu$ is the set $\X$ of measurable functions $u: \mathbb R^N \to \mathbb R$ such that
\begin{equation*}
    u \equiv 0 \quad \text{ in } \mathbb R^N \setminus \Omega
\end{equation*}
and
\begin{equation*}
    \intzo [u]_s^2 d \mu^+(s) < + \infty.
\end{equation*}
Indeed, the following lemma holds.

\begin{lemma}[{\cite[Lemma 2.2]{DipierroPereraSportelliValdinoci2024}}]
    \label{lem:DpPSV_lemma_2.2}
    If $\mu^+$ satisfies  \eqref{eq:hyp_mu_a}, then $\X$ is a Hilbert space with the norm
    \begin{equation*}
        \|u\|_{\X} \coloneqq \left(\intzo [u]_s^2 \ d\mu^+(s) \right)^{1/2}.
    \end{equation*}
\end{lemma}

Now we denote by $\langle \cdot, \cdot \rangle$ the inner product in $\X$ induced by the norm $\|\cdot\|_{\X}$:
\begin{equation*}
    \langle u, v \rangle \coloneqq  \intzo \left(\cns \intrn \intrn \frac{(u(x) - u(y))(v(x) - v(y))}{|x - y|^{N + 2s}} \ dx \ dy \right) \ d\mu^+(s) \quad \forall u, v \in \X.
\end{equation*}
It is also useful to define the following bilinear form:
\begin{equation*}
    (u, v) \coloneqq \int_{[0, \bar s]} \left(\cns \intrn \intrn \frac{(u(x) - u(y))(v(x) - v(y))}{|x - y|^{N + 2s}} \ dx \ dy \right) \ d\mu^-(s) \quad \forall u, v \in \X,
\end{equation*}
in fact
\begin{equation*}
    \intzo [u]_s^2 \ d\mu(s) = \langle u, u \rangle - (u, u).
\end{equation*}

The following result deals with the possibility of {\em reabsorbing} $\mu^-$ in the signed measure $\mu$.

\begin{prop}[{\cite[Proposition 2.3]{DipierroPereraSportelliValdinoci2024}}]
    \label{prop:DpPSV_prop_2.3}
    If \eqref{eq:hyp_mu_a}-\eqref{eq:hyp_mu_c} hold, then there exists a constant $c_0 = c_0(N, \Omega)>0$ such that 
    \begin{equation*}
        \int_{[0, \bar s]} \gag^2 \ d\mu^-(s) \leq c_0 \gamma \int_{[\bar s, 1]} \gag^2  d \mu(s) \quad \forall u \in \X.
    \end{equation*}
\end{prop}

\begin{remark}
    \label{rem:equivalent_norms}
    If $\gamma$ is small, recall \eqref{eq:hyp_mu_c}, then the bilinear form
    \begin{equation}
        \label{eq:equivalent_inner_product}
        u, v \in \X \mapsto \langle u, v \rangle - (u, v) \in \mathbb R
    \end{equation}
    induces a norm in $\X$ which is equivalent to the one given in Lemma \ref{lem:DpPSV_lemma_2.2}. Indeed, in view of Proposition \ref{prop:DpPSV_prop_2.3}, we have
    \begin{equation*}
        \langle u, u \rangle \geq \langle u, u \rangle - (u, u) \geq (1 - c_0 \gamma) \langle u, u \rangle.
    \end{equation*}
    In particular, it follows that the bilinear form \eqref{eq:equivalent_inner_product} is coercive and, therefore, gives rise to an equivalent inner product in $\X$. We denote by $\| \cdot \|$ this equivalent norm in $\X$, i.e.,
    \begin{equation*}
        \|u\|^2 = \langle u, u \rangle - (u, u), \quad u \in \X.
    \end{equation*}
\end{remark}

In the next proposition, some crucial embedding properties are stated.

\begin{prop}[{\cite[Proposition 2.4]{DipierroPereraSportelliValdinoci2024}}]\label{prop:sharp}    Suppose that $\mu = \mu^+ - \mu^-$ satisfies \eqref{eq:hyp_mu_a}-\eqref{eq:hyp_mu_c}. Let $s_\sharp \in [\bar s, 1]$ be such that $\mu^+([s_\sharp, 1]) > 0$. Then, there exists a constant $\bar c = \bar c(N, \Omega, s_\sharp)>0$ such that
    \begin{equation*}
        [u]_{s_\sharp} \leq \bar c \left( \intzo [u]_s^2 \ d\mu^+(s) \right)^{1/2}.
    \end{equation*}
    In particular, $\X$ is continuously embedded in $L^p(\Omega)$ for every $p \in [1, \critexp]$ and compactly embedded in $L^p(\Omega)$ for every $p \in [1, \critexp)$.
\end{prop}

By a solution to \eqref{eq:main_problem}, we mean a weak solution, i.e., a function $u \in \X$ such that for all $v \in \X$ it holds
\begin{equation*}
    \langle u, v \rangle - (u, v) = \bar \lambda \int_\Omega u v  \ dx + \int_\Omega f(x, u) v \ dx,
\end{equation*}
that is, 
\begin{align}
    \intzo & \left(\cns \intrn \intrn \frac{(u(x) - u(y))(v(x) - v(y))}{|x - y|^{N + 2s}} \ dx \ dy \right) \ d\mu^+(s) \nonumber \\
    & - \int_{[0, \bar s]} \left(\cns \intrn \intrn \frac{(u(x) - u(y))(v(x) - v(y))}{|x - y|^{N + 2s}} \ dx \ dy \right) \ d\mu^-(s) \nonumber \\
    & = \bar \lambda \int_\Omega u v \ dx + \int_\Omega f(x, u) v \ dx. \nonumber 
\end{align}

\subsection{On the eigenvalues of the operator $\Amu$}
By a solution to \eqref{eq:eigenvalue_problem}, we mean a weak solution, that is, a function $u \in \X$ such that for all $v \in \X$ it holds
\begin{align}
    \langle u, v \rangle - (u, v) = \lambda \int_\Omega u v \ dx, \nonumber 
\end{align}
for some $\lambda \in \mathbb R$.

It has been shown that the operator $\Amu$ satisfies conditions that are adequate for a spectral theory akin to that of the (classical or fractional) Laplace operator (compare with \cite{BartoloCandelaSalvatore2012} and \cite{BartoloMolicaBisci2015}). Namely, the following holds:

\begin{theorem}[{\cite[Theorem 1.3]{DipierroLippiSportelliValdinoci2025spectral}}]
    \label{thm:spectral}
    Let $\mu$ satisfy \eqref{eq:hyp_mu_a}-\eqref{eq:hyp_mu_c}, and assume that $N > 2s_{\sharp}$. Let $\Omega \subset \mathbb R^N$ be an open bounded set with Lipschitz boundary. Then, there exists $\gamma_0 > 0$ depending only on $N$ and $\Omega$ such that if $\gamma \in [0, \gamma_0]$, then:
    \begin{enumerate}
        \item there exists a sequence $\{\lambda_k\}_{k \in \mathbb N}$ of eigenvalues for \eqref{eq:eigenvalue_problem}, with a corresponding sequence of eigenfunctions $\{e_k\}_{k \in \mathbb N}$, such that
        \begin{equation*}
            0 < \lambda_1 \leq \lambda_2 \leq \ldots \leq \lambda_k \leq \ldots \to + \infty;
        \end{equation*}
    
        \item for any $k \in \mathbb N$, it holds
        \begin{equation}
            \label{eq:Rayleigh_k}
            \lambda_{k + 1} = \min_{u \in \P_{k + 1} \setminus \{0\}} \frac{\langle u, u \rangle - (u, u)}{\|u\|_2^2},
        \end{equation}
        where
        \begin{equation*}
            \P_{k + 1} \coloneqq \{u \in X_0(\Omega) \ : \ \langle u, e_j \rangle - (u, e_j) = 0 \ \forall j = 1, \ldots, k\};
        \end{equation*}

        \item the eigenfunction $e_k$ attains the minimum in \eqref{eq:Rayleigh_k};

        \item the sequence of eigenfunctions $\{e_k\}_{k \in \mathbb N}$ constitutes an orthonormal basis of $L^2(\Omega)$ and an orthogonal basis of $X_0(\Omega)$;

        \item each eigenvalue $\lambda_k$ has finite multiplicity.
    \end{enumerate}
\end{theorem}

\begin{remark}
    \label{rem:on_H_k}
    Let
    \begin{equation*}
        \H_k = \spanned\{e_1, \ldots, e_k\}
    \end{equation*}
    be the subspace of $X_0(\Omega)$ spanned by the first $k$ eigenfunctions. It is clear that
    \begin{equation*}
        X_0(\Omega) = \H_k \oplus \P_{k + 1}.
    \end{equation*}
    Moreover, 
    \begin{equation*}
        \left\{\frac{1}{\sqrt{\lambda_k}} e_k\right\}_{k \in \mathbb N}
    \end{equation*}
    is an orthonormal basis for $X_0(\Omega)$. Therefore, for all $u \in \H_k$, it holds that
    \begin{align}
        \|u\|^2 
        & = \sum_{j = 1}^k \frac{1}{\lambda_j} (\langle u, e_j \rangle - (u, e_j))^2 \nonumber \\
        &  \leq \lambda_k \sum_{j = 1}^k \left(\int_\Omega u e_j \ dx \right)^2 \nonumber \\
        & = \lambda_k \|u\|_2^2. \nonumber
    \end{align}
\end{remark}

\subsection{An abstract critical point theorem}
Here, we outline the main classical tools for dealing with asymptotically linear problems. Namely, we recall an abstract critical point theorem, proved in \cite{BartoloBenciFortunato1983}, based on pseudo-index theories, introduced in \cite{Benci1982}, to study the multiplicity of critical points of strongly indefinite functionals in the presence of symmetries.

Let $E$ be a real Banach space with norm $\|\cdot\|_E$, $E'$ its dual space and $I \in C^1(E, \mathbb R)$.

\begin{definition}[Palais-Smale condition]
    \label{def:PS}
    The functional $I$ satisfies the Palais-Smale condition, (PS) in short, at the level $c \in \mathbb R$, if any sequence $\{u_n\}_{n \in \mathbb N}$ in $E$ such that
    \begin{align}
        & I(u_n) \to c \quad \text{ as } n \to +\infty \nonumber \\
        & I'(u_n) \to 0 \text{ in } E' \quad \text{ as } n \to +\infty \nonumber
    \end{align}
    converges in $E$, up to subsequences. Such a sequence $\{u_n\}_{n \in \mathbb N}$ is said to be a Palais-Smale sequence. In general, if $- \infty \leq a < b \leq + \infty$, $I$ satisfies (PS) in $(a, b)$ if it satisfies (PS) for every $c \in (a, b)$.
\end{definition}

The following abstract critical point theorem, whose proof  is based on the pseudo-index related to the genus (see \cite{Benci1982} for more details), is key to our proof of 
Theorem \ref{thm:main}.		
\begin{theorem}[\cite{BartoloBenciFortunato1983}, Theorem 2.9]\label{thm:pseudo_index}
Let $I\in C^1(E,\mathbb R)$ and assume that:
\begin{enumerate}
\item $I$ is even;
\item $I$ satisfies (PS) in $\mathbb R$;
\item there exist two closed subspaces $V,W\subset X$ such that  $\dim V<+\infty, {\rm codim}\, W<+\infty$ and two constants $c_0,c_\infty$, such that $c_\infty>c_0$, verifying the following assumptions:
\begin{itemize}
\item $I(u)\geq c_0 $ on $S_\rho\cap W$ (resp. on $S_\rho\cap V$), 	where $S_\rho=\{u\in E: \|u\|_X=\rho\}$;
\item $I(u)\leq c_\infty $ on $V$  (resp. on $W$).
\end{itemize}
\end{enumerate}
If, moreover, $\dim V >  {\rm codim}\, W$, then $I$ has at least $\dim V -  {\rm codim}\, W$ distinct pairs of critical points whose corresponding critical values belong to $[c_0,c_\infty]$.
\end{theorem}

\section{Existence and multiplicity results}
\label{sec:existence}

By \eqref{eq:hyp_f_1} and \eqref{eq:hyp_f_2}, we have that for every $\varepsilon > 0$ there exists $a_\varepsilon > 0$ such that 
\begin{equation}
    \label{eq:B_MB_3.1}
    |f(x, t)| \leq \varepsilon |t| + a_\varepsilon, \quad \forall t \in \mathbb R \ \text{ almost everywhere in } \Omega.
\end{equation}

The weak solutions of \eqref{eq:main_problem} can be found as critical points of the functional
\begin{equation*}
    J(u) = \frac{1}{2} \|u\|^2 - \frac{\bar \lambda}{2} \int_\Omega u^2 \ dx - \int_\Omega F(x, u) \ dx, \quad u \in \X,
\end{equation*} 
where $\|\cdot\|$ has been defined in Remark \ref{rem:equivalent_norms} and, as usual, $ F(x, t) = \int_0^{t} f(x, s) \ ds.$

\begin{prop}
    \label{prop:Palais_Smale}
    Assume that \eqref{eq:hyp_f_1}-\eqref{eq:hyp_f_2} hold. If $\bar \lambda \not \in \sigma(\Amu)$, then the functional $J$ satisfies (PS) in $\mathbb R$.
\end{prop}

\begin{proof}
    Let $\{u_n\}_{n \in \mathbb N}$ be a (PS) sequence, i.e., it holds that
    \begin{align}
        &  \{J(u_n)\}_{n \in \mathbb N} \text{ is bounded}; \nonumber \\
        & \lim_{n \to +\infty} J'(u_n) = 0 \quad \text{ in the dual space } (\X)'. \label{eq:PS_b}
    \end{align}
    In particular, due to \eqref{eq:PS_b} it holds that
    \begin{equation}
        \label{eq:B_MB_3.4}
        \langle u_n, v \rangle - (u_n, v) - \bar \lambda \int_\Omega u_n v \ dx - \int_\Omega f(x, u_n) v \ dx = o(1) \quad \forall v \in \X,
    \end{equation}
    where $o(1)$ denotes an infinitesimal sequence.

    We claim that $\{u_n\}_{n \in \mathbb N}$ is bounded in $\X$. Indeed, arguing by contradiction, suppose that
    \begin{equation*}
        \|u_n\| \to + \infty \quad \text{ as } n \to +\infty.
    \end{equation*}
    For each $n$, we set
    \begin{equation*}
        w_n \coloneqq \frac{u_n}{\|u_n\|},
    \end{equation*}
    so that the sequence $\{w_n\}_{n \in \mathbb N}$ is bounded in $\X$. Since $\X$ is a Hilbert space, then there exists some $w \in \X$ such that
    \begin{equation*}
        w_n \rightharpoonup w \quad \text{ weakly in } \X
    \end{equation*}
    as $n \to +\infty$. Since the embedding of $\X$ into $L^2(\Omega)$ is compact (Proposition \ref{prop:sharp}), then also
    \begin{equation*}
        w_n \to w \quad \text{ strongly in } L^2(\Omega)
    \end{equation*}
    as $n \to +\infty$.

    Taking $w_n - w$ as test function in \eqref{eq:PS_b}, we obtain
    \begin{equation*}
        \langle u_n, w_n - w \rangle - (u_n, w_n - w) = \bar \lambda \int_\Omega u_n (w_n - w) \ dx + \int_\Omega f(x, u_n) (w_n - w) \ dx + o(1).
    \end{equation*}
    For each $n \in \mathbb N$, we divide by $\|u_n\|$ to obtain
    \begin{equation*}
        \langle w_n, w_n - w \rangle - (w_n, w_n - w) = \bar \lambda \int_\Omega w_n (w_n - w) \ dx + \int_\Omega \frac{f(x, u_n)}{\|u_n\|} (w_n - w) \ dx + o(1).
    \end{equation*}
    Now, since $w_n \to w$ strongly in $L^2(\Omega)$, we have
    \begin{equation*}
        \left|\int_\Omega w_n (w_n - w) \ dx \right| \leq \|w_n\|_2 \|w_n - w\|_2 = o(1),
    \end{equation*}
    because $\{w_n\}_{n \in \mathbb N}$, being bounded in $\X$, is bounded in $L^2(\Omega)$ too. On the other hand, by \eqref{eq:B_MB_3.1}, for any $\varepsilon > 0$ we have
    \begin{equation*}
        \left|\int_\Omega \frac{f(x, u_n)}{\|u_n\|}(w_n - w) \ dx \right| \leq \varepsilon \|w_n\|_2 \|w_n - w\|_2 + a_\varepsilon \frac{\|w_n - w\|_1}{\|u_n\|} = o(1).
    \end{equation*}
    Therefore, we conclude that
    \begin{equation*}
        \langle w_n, w_n - w \rangle - (w_n, w_n - w) = o(1).
    \end{equation*}
    But, since
    \begin{align}
        o(1) 
        & = \langle w_n, w_n - w \rangle - (w_n, w_n - w) \nonumber \\
        & = \langle w_n - w, w_n - w \rangle - (w_n - w, w_n - w) \nonumber \\
        & \quad + \langle w, w_n - w \rangle - (w, w_n - w) \nonumber \\
        & = \|w_n - w\|^2 + o(1), \nonumber
    \end{align}
    it follows that
    \begin{equation*}
        w_n \to w \quad \text{ strongly in } \X
    \end{equation*}
    and therefore $w \neq 0$.

    Dividing \eqref{eq:B_MB_3.4} by $\|u_n\|$, we obtain
    \begin{equation*}
        \langle w_n, v \rangle - (w_n, v) = \bar \lambda \int_\Omega w_n v \ dx - \int_\Omega \frac{f(x, u_n)}{\|u_n\|} v \ dx + o(1), \quad \forall v \in \X.
    \end{equation*}
    As before, since $\|u_n\| \to + \infty$ as $n \to +\infty$, then
    \begin{equation*}
        \int_\Omega \frac{f(x, u_n)}{\|u_n\|} v \ dx = o(1) ,\quad \forall v \in \X.
    \end{equation*}
    Then, passing to the limit, we obtain
    \begin{equation*}
        \langle w, v \rangle - (w, v) = \bar \lambda \int_\Omega w v \ dx ,\quad \forall v \in \X,
    \end{equation*}
    implying that $w$ is an eigenfunction with corresponding eigenvalue $\bar \lambda$. This contradicts the assumptions that $\bar \lambda \notin \sigma(\Amu)$. 
    
    The sequence $\{u_n\}_{n \in \mathbb N}$ is therefore bounded in $\X$. Since $\X$ is a Hilbert space, then there exists some $u \in \X$ such that
    \begin{align}
        & u_n \rightharpoonup u \quad \text{ weakly in } \X \nonumber \\
        & u_n \to u \quad \text{ strongly in } L^2(\Omega) \nonumber
    \end{align}
    as $n \to +\infty$. Hence
    \begin{equation*}
        J'(u_n)[u_n - u] \to 0 \quad \text{ as } n \to +\infty.
    \end{equation*}
    By the same reasoning as before, we obtain that
    \begin{equation*}
        u_n \to u \quad \text{ strongly in } \X,
    \end{equation*}
    which completes the proof.
\end{proof}

\subsection*{\textit{Proof of Proposition \ref{prop:existence_via_saddle_point}.}} 
    The proof follows by a standard application of variational methods (see, e.g., \cite{Rabinowitz1986, BadialeSerra2011}).

    We begin by noting that, by \eqref{eq:B_MB_3.1}, 
    \begin{equation*}
        |F(x, t)| \leq C_1(1 + t^2), \quad \forall t \in \mathbb R \text{ almost everywhere in } \Omega.
    \end{equation*}
    Now, by the variational characterization of the eigenvalues of $\Amu$ (see Theorem \ref{thm:spectral}), we have that given any $k \in \mathbb N$, it holds
    \begin{equation*}
        \|u\|_2^2 \leq \frac{1}{\lambda_k + 1} \|u\|^2, \quad \forall u \in \P_{k + 1}.
    \end{equation*}
    It follows that
    \begin{align}
        J(u)
        & = \frac{1}{2}\|u\|^2 - \frac{\bar \lambda}{2} \int_\Omega u^2 \ dx - \int_\Omega F(x, u) \ dx \nonumber \\
        & \geq \frac{1}{2} \left(1 - \frac{\bar \lambda}{\lambda_{k + 1}} - \frac{C_2}{\lambda_{k + 1}}\right) \|u\|^2 - C_3, \quad \forall u \in \P_{k + 1} \nonumber
    \end{align}
    where $C_2$ and $C_3$ do not depend on $k$. Therefore, provided that $k$ is large enough, we have that $J$ is bounded from below in $\P_{k + 1}$. In other words, there exists $\beta \in \mathbb R$ such that
    \begin{equation*}
        J|_{\P_{k + 1}} \geq \beta.
    \end{equation*}

    If $\bar \lambda < \lambda_1$, then there exists a minimizer for the functional $J$ in $\X$, since $J$ is weakly lower-semicontinuous and coercive.

    On the other hand, suppose that there for some $k \in \mathbb N$ it holds
    \begin{equation*}
        \lambda_k < \bar \lambda < \lambda_{k + 1}.
    \end{equation*}
    Let $\varepsilon > 0$ be such that $\lambda_k + \varepsilon < \bar \lambda$. By \eqref{eq:B_MB_3.1}, there exists a constant $C_\varepsilon > 0$ such that
    \begin{equation*}
        J(u) \leq \frac{1}{2} \|u\|^2 - \frac{\bar \lambda}{2} \|u\|_2^2 + \frac{\varepsilon}{2} \|u\|_2^2 + C_\varepsilon \|u\|_2, \quad \forall u \in \X,
    \end{equation*}
    and therefore, by Remark \ref{rem:on_H_k},
    \begin{equation*}
        J(u) \leq \frac{1}{2} (\lambda_k + \varepsilon - \bar \lambda)\|u\|_2^2 + C_\varepsilon \|u\|_2, \quad \forall u \in \H_k.
    \end{equation*}
    Therefore we can find $\rho > 0$ big enough such that
    \begin{equation*}
        J|_{S_\rho^k} \leq \alpha < \beta,
    \end{equation*}
    where $S_\rho^k$ is the sphere of radius $\rho$ in the finite dimensional space $\H_k$.

    Therefore, all the geometric assumptions of the Saddle Point Theorem (\cite[Theorem 4.6]{Rabinowitz1986}) are satisfied. Since the functional $J$ satisfies (PS), then $J$ has at least one critical point, i.e., \eqref{eq:main_problem} admits at least one nontrivial solution.
    \qed

\subsection*{\textit{Proof of Theorem \ref{thm:main}}}
    The proof is divided into three steps.
    \proofstep{1}{} We claim that there exist $\rho > 0$ and $c_0 > 0$ such that 
    \begin{equation}
        \label{eq:B_MB_3.19}
        J(u) \geq c_0, \quad \forall u \in S_\rho \cap \P_h,
    \end{equation}
    where
    \begin{equation*}
        S_\rho \coloneqq \{u \in \X \ : \ \|u\| = \rho\}.
    \end{equation*}

    To prove the claim, we begin by observing that assumption \eqref{eq:hyp_f_2} yields
    \begin{equation*}
        \lim_{|t| \to + \infty} \frac{F(x, t)}{t^2} = 0
    \end{equation*}
    uniformly almost everywhere in $\Omega$. Similarly, \eqref{eq:hyp_f_3} implies that
    \begin{equation*}
        \lim_{|t| \to 0} \frac{F(x, t)}{t^2} = \frac{\lambda_0}{2}
    \end{equation*}
    uniformly almost everywhere in $\Omega$. It follows that, for every $\varepsilon > 0$, there exist $r_\varepsilon \geq 1$ and $\delta_\varepsilon > 0$ such that
    \begin{align}
        & |F(x, t)| \leq \frac{\varepsilon}{2} t^2, \quad \text{ if } |t| > r_\varepsilon, \nonumber \\
        & \left|F(x, t) - \frac{\lambda_0}{2}t^2 \right| \leq \frac{\varepsilon}{2}t^2, \quad \text{ if } t < \delta_\varepsilon, \nonumber
    \end{align}
    almost everywhere in $\Omega$. Finally, by \eqref{eq:hyp_f_1}, we can choose $q \in [0, \critexp - 2)$ and obtain a constant $\kre > 0$ such that
    \begin{equation*}
        |F(x, t)| \leq \kre |t|^{q + 2}, \quad \text{ if } \delta_\varepsilon \leq |t| \leq r_\varepsilon
    \end{equation*}
    almost everywhere in $\Omega$.

    Combining these estimates, it follows that for every $\varepsilon > 0$ there exists $k_\varepsilon > 0$ such that
    \begin{equation*}
        F(x, t) \leq \frac{\lambda_0 + \varepsilon}{2} t^2 + k_\varepsilon |t|^{q + 2}, \quad \forall t \in \mathbb R,
    \end{equation*}
    almost everywhere in $\Omega$. Whence
    \begin{equation*}
        \int_\Omega F(x, u) \ dx \leq \frac{\lambda_0 + \varepsilon}{2} \|u\|_2^2 + k_\varepsilon \|u\|_{q + 2}^{q + 2}, \quad \forall u \in \X,
    \end{equation*}
    and therefore, since $\X$ is continuously embedded in $L^q(\Omega)$, we obtain
    \begin{equation*}
        J(u) \geq \frac{1}{2} \left(1 - \frac{\bar \lambda + \lambda_0 + \varepsilon}{\lambda_h} \right) \|u\|^2 - k_\varepsilon' \|u\|^{q + 2}, \quad \forall u \in \P_h.
    \end{equation*}
    Therefore, by \eqref{eq:hyp_Lambda}, for a suitable $\varepsilon > 0$ there exists $k_\varepsilon'' > 0$ such that 
    \begin{equation*}
        J(u) \geq k_\varepsilon'' \|u\|^2 - k_\varepsilon' \|u\|^{q + 2}, \quad \forall u \in \P_h.
    \end{equation*}
    By choosing $\rho$ small enough, we obtain \eqref{eq:B_MB_3.19}.

    \proofstep{2}{} We claim that there exists $c_\infty > c_0$ such that 
    \begin{equation*}
        J(u) \leq c_\infty, \quad \forall u \in \H_k.
    \end{equation*}

    To prove the claim, we let $k$ be as in \eqref{eq:hyp_Lambda} and $\varepsilon > 0$ be such that $\lambda_k + \varepsilon < \bar \lambda$, so that, as before,
    \begin{equation*}
        J(u) \leq \frac{1}{2} \|u\|^2 - \frac{\bar \lambda}{2} \|u\|_2^2 + \frac{\varepsilon}{2} \|u\|_2^2 + C_\varepsilon \|u\|_2, \quad \forall u \in \H_k.
    \end{equation*}
    It becomes clear that $J(u) \to - \infty$ as $\|u\| \to + \infty$ in $\H_k$, and the claim follows.

    \proofstep{3}{} 
    Now we can conclude the proof. Indeed, $J$ is even and  by Lemma \ref{prop:Palais_Smale} it satisfies (PS) in $\mathbb R$; moreover,	taking $W=\mathbb{P}_{h}$ and $V=\mathbb{H}_{k}$, Theorem 	\ref{thm:pseudo_index} applies and $J$ has at least $k-h+1$ distinct pairs of critical points.    
    \qed

We state now one of the several possible corollaries of our main result, as pointed out in Section \ref{sec:intro}.

\begin{corollary}\label{cor:serie}
{\em Let $\sum_{k = 0}^\infty c_k \in (0,+\infty)$,  with $c_k \in\mathbb R$ for all $k \in \mathbb N$, $1\geq s_0>s_1>s_2>\ldots\geq 0$ and assume that there exist $\gamma\geq 0$ and $\bar k\in\mathbb N$ such that $c_k>0$ for all $k\in \{1,\ldots,\bar k\}$ and 
$$
\sum_{k = \bar k + 1}^\infty c_k \leq \gamma \sum_{k = 0}^{{\bar k}} c_k .
$$
Then, under the assumptions of Theorem \ref{thm:main}, with $\{\lambda_j\}_{j\in\mathbb N}$ sequence of Dirichlet eigenvalues of the operator
$A_\mu=\sum_{k = 1}^\infty c_k (- \Delta)^{s_k}$, there exists $\gamma_0 > 0$, depending on $N$, $\Omega$, and $\bar \lambda$, such that if $\gamma \in [0, \gamma_0]$ and $\bar \lambda \notin \sigma(\Amu)$, then 
 \begin{equation*}
    \left\{
    \begin{array}{rcll}
        \sum_{k = 1}^\infty c_k (- \Delta)^{s_k} u- \bar \lambda u & = & f(x, u) & \quad \text{ in } \Omega \\
        u & = & 0 & \quad \text{ in } \mathbb R^N \setminus \Omega
    \end{array}
    \right.
\end{equation*} 
 admits at least $k - h + 1$ pairs of nontrivial weak solutions.}

\end{corollary}

\section*{Acknowledgements}
We are grateful to Caterina Sportelli for pointing out a flaw in an earlier version of the paper and for her very useful comments.

R. B.  is  partially supported by the European Union - Next Generation EU - PRIN 2022 PNRR {``P2022YFAJH Linear and Nonlinear PDE's: 
		New directions and Application''} and \\by  MUR  under the Programme ``Department of Excellence'' Legge 232/2016  (Grant CUP - D93C23000100001).		

D. G. A. and G. M. B. are partially supported by the European Union - NextGenerationEU within the framework of PNRR  Mission 4 - Component 2 - Investment 1.1 under the Italian Ministry of University and Research (MUR) program PRIN 2022 - Grant number 2022BCFHN2 - Advanced theoretical aspects in PDEs and their applications - CUP: H53D23001960006.

All authors are also partially supported by Gruppo Nazionale per l'Analisi Matematica, la Probabilità e le loro Applicazioni (GNAMPA) of the Istituto Nazionale di Alta Matematica (INdAM), Italy.

\bibliographystyle{acm}
\bibliography{ref_math}

\end{document}